\documentclass[12pt]{article}

\usepackage{amsmath}
\usepackage{amssymb,amsfonts,amsthm}
\usepackage{tikz,color,pgf,graphicx,url}
\usepackage{enumerate}

\usepackage[utf8]{inputenc}
\usepackage[english]{babel}
\usepackage{graphicx}

\newtheorem{theorem}{Theorem}[section]

\newtheorem{claim}[theorem]{Claim}

\newtheorem{definition}[theorem]{Definition}

\newtheorem{example}[theorem]{Example}

\DeclareMathOperator{\crit}{cr}
\renewcommand{\S}{\mathcal{S}}
\newcommand{\T}{\mathcal{T}}
\DeclareMathOperator{\sig}{sig}

\usepackage{amssymb}
\newcommand*{\myproofname}{Proof}
\newenvironment{myproof}[1][\myproofname]{\begin{proof}[#1]}{\end{proof}}

\begin{document}

\title{Domination and independence number of large $2$-crossing-critical graphs}

\author{Vesna Ir\v si\v c$^{a,b,c}$\thanks{Email: \texttt{vesna\_irsic@sfu.ca}}
\and Maru\v sa Lek\v se$^{a}$\thanks{Email: \texttt{marusa.lekse@student.fmf.uni-lj.si}}
\and Mihael Pa\v cnik$^{a}$\thanks{Email: \texttt{mihael.pacnik@student.fmf.uni-lj.si}}
\and Petra Podlogar$^{a}$\thanks{Email: \texttt{petra.podlogar@student.fmf.uni-lj.si}}
\and Martin Pra\v cek$^{a}$\thanks{Email: \texttt{martin.pracek@student.fmf.uni-lj.si}}
}
\maketitle

\begin{center}
$^a$ Faculty of Mathematics and Physics, University of Ljubljana, Slovenia\\
\medskip

$^b$ Institute of Mathematics, Physics and Mechanics, Ljubljana, Slovenia\\
\medskip

$^c$ Department of Mathematics, Simon Fraser University, Burnaby, BC, Canada\\
\medskip
\end{center}

\begin{abstract}
After $2$-crossing-critical graphs were characterized in 2016, their most general subfamily, large $3$-connected $2$-crossing-critical graphs, has attracted separate attention. This paper presents sharp upper and lower bounds for their domination and independence number.
\end{abstract}


\noindent
{\bf AMS Subj.\ Class.\ (2020)}: 05C10, 05C62, 05C69.

\section{Introduction}

The \emph{crossing number} $\crit(G)$ of a graph $G$ is the smallest number of edge crossings in a drawing of $G$ in a plane. The topic has been widely studied, see for example~\cite{chimani+2022, clancy+2020, shahrokhi+1995, silva+2019, szekely2008}. A graph $G$ is \emph{$k$-crossing-critical} if $\crit(G) \geq k$, but every proper subgraph $H$ of $G$ has $\crit(H) < k$. Note that subdividing an edge or its inverse operation does not affect the crossing number of a graph. Thus we can restrict our studies to graphs without degree $2$ vertices. Under this restriction, Kuratowski's Theorem tells us that the only $1$-crossing-critical graphs are $K_5$ and $K_{3,3}$. The classification of $2$-crossing-critical graphs has been of interest since the 1980s. Partial results on the topic have been reported in~\cite{bloom+1983, ding+2011, kochol1987, richter1988, siran1984}, and some related results can be found in~\cite{beaudou+2013, hlineny2003, leanos+2008}. Crossing numbers of graphs with a tile structure have been studied in~\cite{pinontoan+2004, pinontoan+2003}. Finally, Bokal, Oporewski, Richter, and Salazar~\cite{ccc} provided an almost complete characterization of $2$-crossing-critical graphs. In particular, they describe a tile structure of large $3$-connected $2$-crossing-critical graphs (i.e., all but finitely many $3$-connected $2$-crossing-critical graphs). Recently, the degree properties of crossing-critical graphs have been studied in~\cite{bokal+2019b, bokal+2019, hlineny+2019}.

The above-mentioned large $3$-connected $2$-crossing-critical graphs have since attracted separate attention, see~\cite{bokal+2021+, bokal+2021, ham}. In~\cite{bokal+2021, ham}, the Hamiltonicity of these 
graphs is discussed, and the number of all Hamiltonian cycles is determined. In~\cite{bokal+2021+}, several additional properties of large $3$-connected $2$-crossing-critical graphs have been studied. In particular, the number of vertices and edges can be determined from the signature of a graph, and several results regarding their chromatic number, chromatic index, and tree-width are presented. In the present paper, we extend the studies of large $3$-connected $2$-crossing-critical graphs to their domination and independence number.

The rest of the paper is organized as follows. In the next section, necessary definitions and known results are listed. In Section~\ref{sec:domination}, the sharp upper and lower bounds for the domination number of large $3$-connected $2$-crossing-critical graphs are given, while in Section~\ref{sec:independence} analogous results are proved for their independence number.

\section{Preliminaries}
\label{sec:preliminaries}

Let $G$ be a graph. Its vertex set is denoted by $V(G)$ and its edge set by $E(G)$. The \emph{(open) neighborhood} of a vertex $v \in V(G)$ is $N(v) = \{ u \in V(G); \; uv \in E(G) \}$ and the \emph{closed neighborhood} of $v$ is $N[v] = \{v\} \cup N(v)$. Similarly, for $D \subseteq V(G)$, $N[D] = \bigcup_{v \in D} N[v]$ is the closed neighborhood of a subset of vertices $D$. Note also that $[n] = \{1, \ldots, n\}$ and that a reversed sequence of a sequence $a$ is denoted by $\overline{a}$.

We now recall the definitions of the domination number and independence number. 

\begin{definition}
Let $G$ be a graph. A subset $D\subseteq V(G)$ \emph{dominates} the set of vertices $X \subseteq V(G)$ if $X \subseteq N[D]$. If $X = V(G)$, then we say that $D$ \emph{dominates} the graph $G$. The size of the smallest dominating set is called the \emph{domination number} of the graph $G$ and it is denoted by $\gamma(G)$.
\end{definition}

\begin{definition}
Let $G$ be a graph. A subset $X\subseteq V(G)$ is \emph{independent} if none of the vertices from $X$ are adjacent. The \emph{independence number} $\alpha(G)$ of the graph $G$ is the size of the largest independent set.
\end{definition}

In the rest of the section, we recall the characterization of $2$-crossing-critical graphs and provide the necessary definitions which help us describe large $3$-connected $2$-crossing-critical graphs, i.e., graphs studied in this paper. Note that vertices of degrees $1$ and $2$ do not affect the crossing number, thus the assumption that the minimum degree is at least $3$ is reasonable.

\begin{theorem}[\cite{ccc}]
\label{thm:characterization}
Let $G$ be a $2$-crossing-critical graph with a minimum degree of at least $3$. Then one of the following holds.
\begin{enumerate}[(i)]
    \item $G$ is $3$-connected, contains a subdivision of $V_{10}$, and has a very particular twisted M\"obius band tile structure, with each tile isomorphic to one of $42$ possibilities. 
    \item $G$ is $3$-connected, does not have a subdivision of $V_{10}$, and has at most $3$ million vertices.
    \item $G$ is not $3$-connected and is one of $49$ particular examples.
    \item $G$ is $2$- but not $3$-connected and is obtained from a $3$-connected $2$-crossing-critical graph by replacing digons with digonal paths. 
\end{enumerate}
\end{theorem}

In the present paper, we study graphs from~$(i)$, i.e., $3$-connected $2$-crossing-critical graphs that contain a subdivision of $V_{10}$. Since $3$-connected $2$-crossing-critical graphs that do not contain a subdivision of $V_{10}$ have at most $3$ million vertices, we may call graphs from Theorem~\ref{thm:characterization}$(i)$ \emph{large $3$-connected $2$-crossing-critical graphs} or \emph{large $3$-con $2$-cc graphs} for short. This abbreviation is used throughout the paper. Note that it would also be interesting to study other subclasses of graphs, especially graphs from $(iv)$. However, like in~\cite{bokal+2021+}, we restrict our studies to graphs from $(i)$.

To understand the tile structure of large $3$-con $2$-cc graphs, we need the following definitions.

\begin{definition}
\label{def:tiles}
\begin{enumerate}
    \item A \emph{tile} is a triplet $T = (G, \lambda, \rho)$, where $G$ is a graph and $\lambda, \rho$ are sequences of distinct vertices of $G$, where no vertex of $G$ appears in both $\lambda$ and $\rho$. 
    
    \item A \emph{tile drawing} is a drawing $D$ of $G$ in the unit square $[0,1] \times [0,1]$ for which the intersection of the boundary of the square with $D$ contains precisely the images of the \emph{left wall} $\lambda$ and the \emph{right wall} $\rho$, and these are drawn in $\{0\} \times [0,1]$ and $\{1\} \times [0,1]$, respectively, such that the $y$-coordinates of the vertices are increasing with respect to their orders in the sequences $\lambda$ and $\rho$. 
    
    \item The tiles $T = (G, \lambda, \rho)$ and $T'=(G', \lambda', \rho')$ are \emph{compatible} if $|\rho| = |\lambda'|$.
    
    \item A sequence $(T_0, \ldots, T_m)$ of tiles is \emph{compatible} if $T_{i-1}$ is compatible with $T_i$ for every $i \in [m]$.
    
    \item The \emph{join} of compatible tiles $(G, \lambda, \rho)$ and $(G', \lambda', \rho')$ is the tile $(G, \lambda, \rho) \otimes (G', \lambda', \rho')$ whose graph is obtained from $G$ and $G'$ by identifying the sequence $\rho$ term by term with the sequence $\lambda'$. The left wall of the obtained tile is $\lambda$ and the right wall is $\rho'$.
    
    \item The \emph{join} $\otimes \mathcal{T}$ of a compatible sequence $\mathcal{T} = (T_0, \ldots, T_m)$ of tiles is defined as $T_0 \otimes \cdots \otimes T_m$.
    
    \item A tile $T$ is \emph{cyclically-compatible} if $T$ is compatible with itself. For a cyclically-compatible tile $T$, the \emph{cyclization} of $T$ is the graph $\circ T$ obtained by identifying the respective vertices of the left wall with the right wall. Cyclization of a cyclically-compatible sequence of tiles is $\circ \mathcal{T} = \circ (\otimes \mathcal{T})$. 
    
    \item Let $T = (G, \lambda, \rho)$ be a tile. The \emph{right-inverted} tile of $T$ is $T^{\updownarrow} = (G, \lambda, \overline{\rho})$. The \emph{left-inverted} tile of $T$ is $^{\updownarrow} T = (G, \overline{\lambda}, \rho)$. The \emph{inverted} tile is $^{\updownarrow} T^{\updownarrow} = (G, \overline{\lambda}, \overline{\rho})$. The \emph{reversed} tile is $T^{\leftrightarrow} = (G, \rho, \lambda)$. 
    
\end{enumerate}
\end{definition}

Note that $\otimes \mathcal{T}$ in Definition~\ref{def:tiles},~6. is well-defined since $\otimes$ is associative.

\begin{definition}
\label{def:pictures&frames}
The set $\S$ of tiles consists of tiles obtained as a combination of two \emph{frames} shown in Figure~\ref{fig:frames} and $13$ \emph{pictures} shown in Figure~\ref{fig:pictures} in such a way that a picture is inserted into a frame by identifying the two geometric squares. (This can mean subdividing the frame's square.) A given picture can be inserted into a frame either with the given orientation or with a $180^\circ$ rotation.
\end{definition}

\begin{figure}[h]
    \centering
    \includegraphics[width=0.6\textwidth]{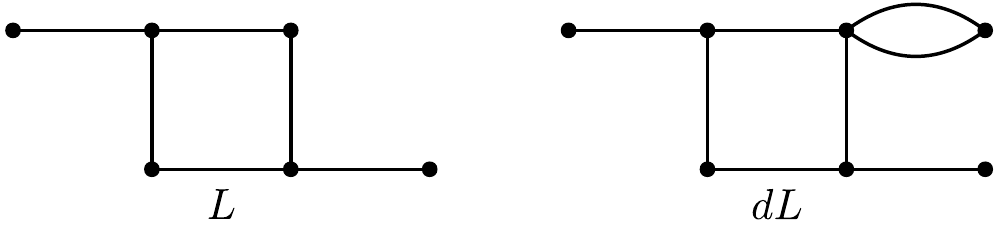}
    \caption{Both possible frames.}
    \label{fig:frames}
\end{figure}

\begin{figure}[h]
    \centering
    \includegraphics[width=0.76\textwidth]{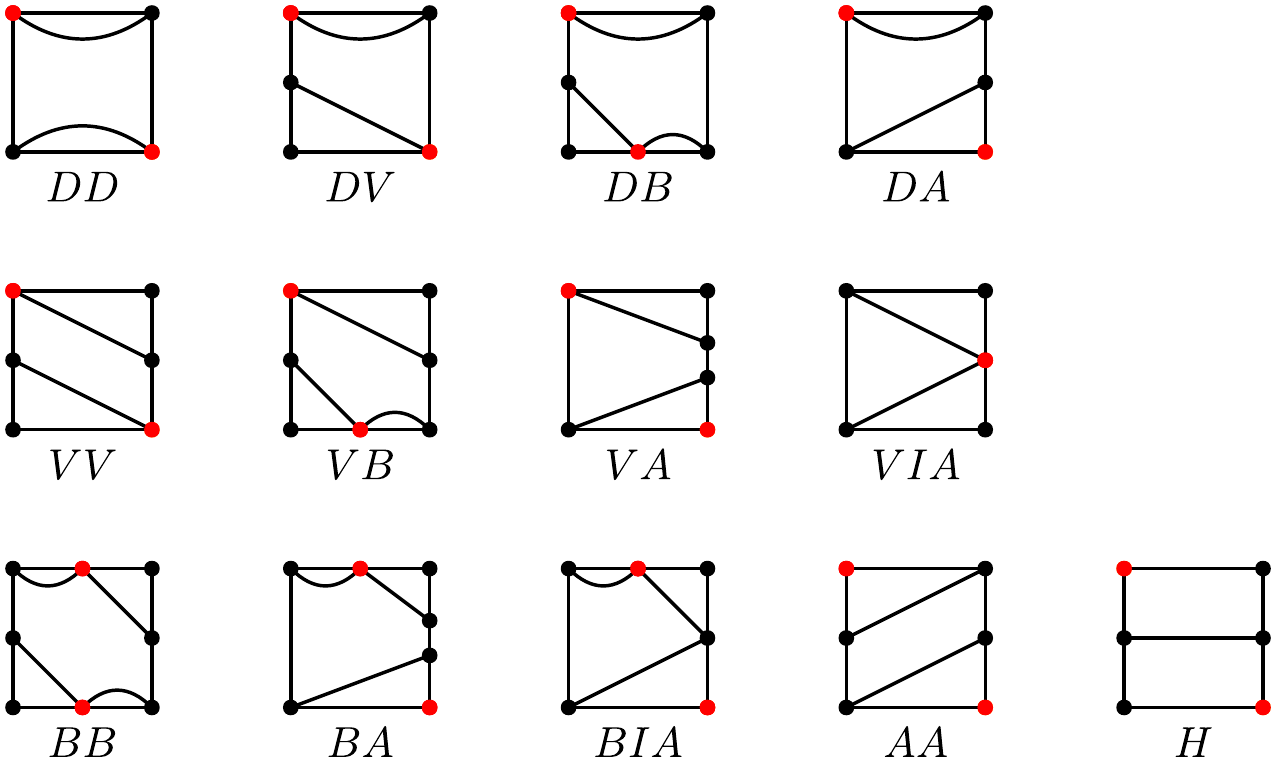}
    \caption{All possible pictures. For later need, the red vertices mark a dominating set of each of them.}
    \label{fig:pictures}
\end{figure}

Note that each picture yields either two or four tiles in $\S$. Altogether the set $\S$ contains $42$ different tiles. For example, in Figure~\ref{fig:tiles} we see that picture $VIA$ yields four different tiles.

\begin{figure}[h]
    \centering
    \includegraphics[width=0.6\textwidth]{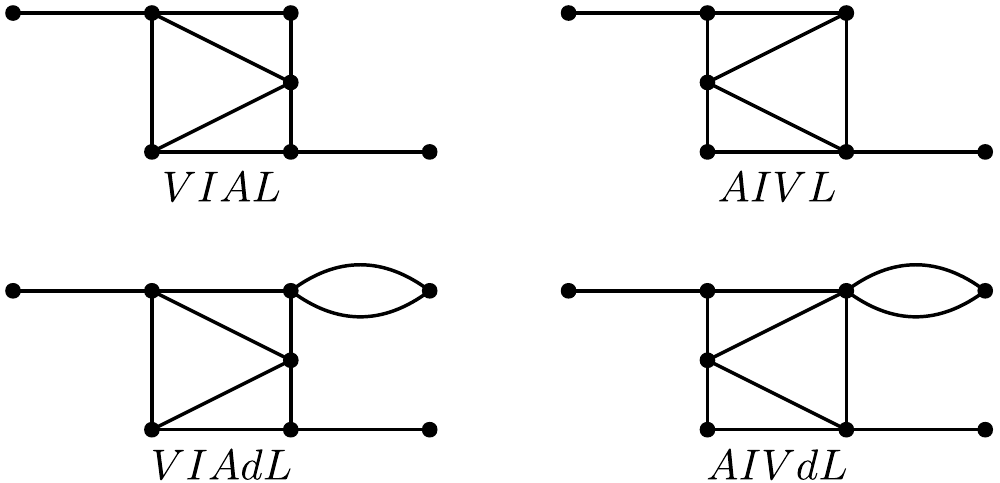}
    \caption{All possible tiles that can be obtained from picture $VIA$.}
    \label{fig:tiles}
\end{figure}

We can now define the tile structure of graphs that are of our interest. Their definition first appeared in~\cite{ccc}. 

\begin{definition}
\label{def:graphs}
The set $\mathcal{T}(\mathcal{S})$ consists of all graphs of the form $\circ((\otimes \mathcal{T})^{\updownarrow})$, where $\mathcal{T}$ is a sequence $(T_0, ^{\updownarrow} T_1^{\updownarrow}, T_2, \ldots, ^{\updownarrow} T_{2m-1}^{\updownarrow}, T_{2m})$, where $m \geq 1$ and $T_i \in \mathcal{S}$ for every $i \in \{0, \ldots, 2m\}$. The obtained vertices of degree $2$ are suppressed. 
\end{definition}


Note that for the case of calculating the domination and independence number of graphs, doubled edges can be replaced with single ones without changing the invariant.

\begin{theorem}(\cite[Theorems 2.18 and 2.19]{ccc})
\label{thm:characterizaton-tiles}
Each graph from $\T(\S)$ is $3$-connected and $2$-crossing-critical. Moreover, all but finitely many $3$-connected $2$-crossing-critical graphs are contained in $\T(\S)$. 
\end{theorem}

Theorem~\ref{thm:characterizaton-tiles} gives a nice representation of large $3$-con $2$-cc graphs, i.e., graphs from Theorem~\ref{thm:characterization} $(i)$.

Graphs from the set $\T(\S)$ can be described as sequences over the alphabet $\Sigma = \{L, d, A, B, D, H, I, V\}$ (see~\cite{ham}). A \emph{signature} of a tile $T$ is $$\sig(T) = P_t \, Id \, P_b \, Fr,$$ where $P_t \in \{A, B, D, H, V\}$ describes the top path of the picture, $Id \in \{I,\emptyset\}$ indicates a possible identifier of the picture, $P_b \in \{ A, B, D, V, \emptyset \}$ describes the bottom path of the picture, and $Fr \in \{ L, dL \}$ describes the frame. Here, $\emptyset$ labels the empty word. See Figure~\ref{fig:frames} for possible signatures of frames ($Fr$), Figure~\ref{fig:pictures} for all possible signatures of pictures ($P_t \, Id \, P_b$), and Figure~\ref{fig:tiles} for an additional example of how to describe a tile with its signature.

For a graph $G \in \T(\S)$, $G = \circ((\otimes \mathcal{T})^{\updownarrow}) = (T_0, ^{\updownarrow} T_1^{\updownarrow}, T_2, \ldots, ^{\updownarrow} T_{2m-1}^{\updownarrow}, T_{2m})$, a signature is defined as $$\sig(G) = \sig(T_0) \sig(T_1) \cdots \sig(T_{2m}).$$ Additionally, $\# X$ denotes the number of occurrences of $X$ in $\sig(G)$, where $X \in \Sigma$. Given a tile $T$, the join of a sequence of $k$ tiles, starting with $T$ and then alternating between $T$ and $T^{\updownarrow}$, is denoted by $k \cdot T$.

\section{Domination number}
\label{sec:domination}

In this section, we present the upper and lower bound for the domination number of large $3$-con $2$-cc graphs, including equality cases for both bounds.

\subsection{Upper bound}
\label{sec:dom-upper}

\begin{theorem}
If $G$ is a large $3$-con $2$-cc graph, then
$$\gamma(G) \leq \#A + \#B + \#D + \#V + 2 \cdot \#H - \#AIV - \#VIA.$$
\label{thm:dom-upper}
\end{theorem}

\begin{proof}
Each vertex lies on at least one picture. Thus, if $D \subseteq V(G)$ dominates all vertices in each picture, then D is a dominating set of G.
Inside each picture we have at least one path (by path we mean the top and the bottom path as in the definition of the signature of a tile).
We can see that the domination of $A$, $B$, $D$, and $V$ requires at least one vertex, while the domination of $H$ requires two vertices. The only exceptions are pictures $AIV$ and $VIA$, where the domination of the picture only requires one vertex and not two, which would be the result of the summation of domination numbers of paths $A$ and $V$. Figure~\ref{fig:pictures} shows all possible pictures with marked smallest dominating sets.


Edges between pictures only add edges between vertices and lower the domination number. This means that the domination number has an upper bound of the sum of domination numbers for individual paths.
\end{proof}

The upper bound from Theorem~\ref{thm:dom-upper} is sharp, which can be seen in the following two examples. They also show that the number of frames $L$ and $dL$ does not affect the upper bound.

\begin{example}
Let $G_1 = n \cdot VBdL$, where $n \geq 3$ is an odd number. Figure~\ref{fig:G1} shows the dominating set of size $2n$, meaning $\gamma(G_1) \leq 2n$. The formula from Theorem~\ref{thm:dom-upper} shows the same, as $\#A + \#B + \#D + \#V + 2 \cdot \#H - \#AIV - \#VIA = 2n$.


\begin{figure}[h]
    \centering
    \includegraphics[width=0.75\textwidth]{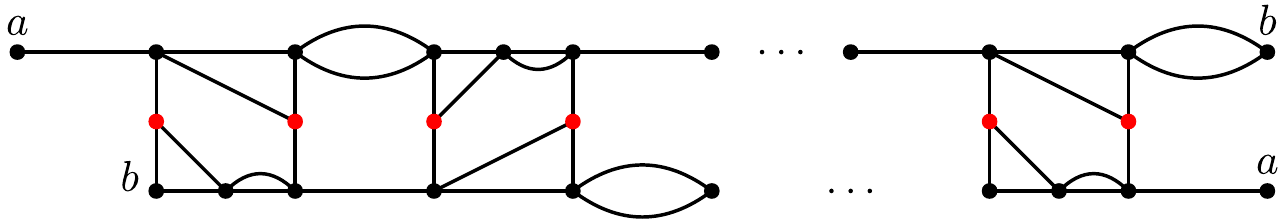}
    \caption{Graph $G_1$ with a marked dominating set of size $2n$. Note that to obtain the desired graph, vertices $a$ are identified, vertices $b$ are identified, and after this vertices of degree $2$ are suppressed. The same simplification of drawings is used for the rest of the paper.}
    \label{fig:G1}
\end{figure}

Assume $\gamma(G_1) < 2n$. The Pigeonhole principle says that there exists at least one picture, which is dominated by at most one vertex. Vertices in the corners of the picture can be dominated by vertices from neighboring pictures. The remaining three inner vertices, which we get from $B$ and $V$ and are painted orange in Figure~\ref{fig:G1_pictureVB}, are yet to be dominated. Since these three vertices cannot be dominated by one vertex, we need at least two vertices to dominate this picture, which leads to a contradiction. Therefore $\gamma(G_1) \geq 2n$.


\begin{figure}[h]
    \centering
    \includegraphics[width=0.095\textwidth]{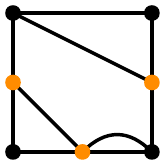}
    \caption{Picture $VB$, where we require that the inner vertices, marked orange, are dominated by one vertex.}
    \label{fig:G1_pictureVB}
\end{figure}

From this, it follows that $\gamma(G_1) = 2n$.
\end{example}

\begin{example}
Let $G_2 = n \cdot AIVL$, where $n \geq 3$ is an odd number. We can find a dominating set of size $n$ (see Figure~\ref{fig:G2}), thus $\gamma(G_2) \leq n$. This also follows from the formula in Theorem~\ref{thm:dom-upper}, as $\#A + \#B + \#D + \#V + 2 \cdot \#H - \#AIV - \#VIA = n$.


\begin{figure}[h]
    \centering
    \includegraphics[width=0.7\textwidth]{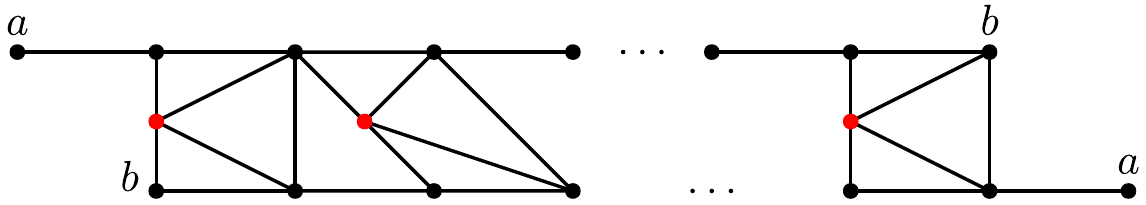}
    \caption{Graph $G_2$ with a marked dominating set of size $n$.}
    \label{fig:G2}
\end{figure}

We next show that $\gamma(G_2) \geq n$. Divide the graph $G_2$ into $n$ disjoint subgraphs, as shown in Figure~\ref{fig:G2_subgraphs}. Each subgraph is induced on the closed neighborhood of the degree $3$ vertex and is isomorphic to the paw graph. The position of degree $3$ vertices in $G_2$ ensures that the obtained $n$ subgraphs are all pairwise disjoint. We notice that the middle vertex of each subgraph (the vertex of degree $3$) can only be dominated by one of the vertices in the same subgraph. Hence we must choose at least one vertex from each one of the $n$ disjoint subgraphs, which means that $\gamma(G_2) \geq n$.

\begin{figure}[h]
    \centering
    \includegraphics[width=0.7\textwidth]{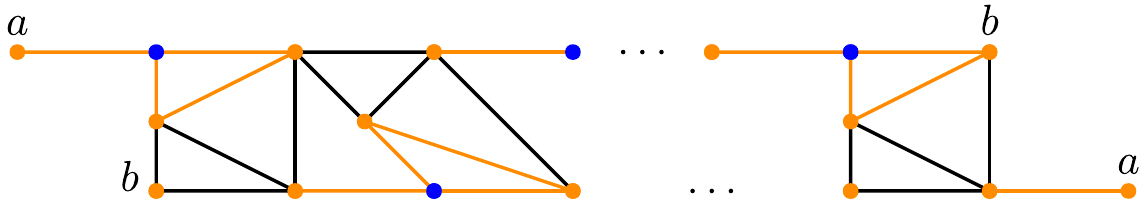}
    \caption{Graph $G_2$ with disjoint subgraphs marked orange and the middle vertex of each subgraph marked blue. Recall that when vertex $a$ is identified, the obtained vertex of degree $2$ is suppressed.}
    \label{fig:G2_subgraphs}
\end{figure}

It follows that $\gamma(G_2) = n$.

\end{example}

\subsection{Lower bound}
\label{sec:dom-lower}

\begin{theorem}
\label{thm:dom-lower}
If $G$ is a large $3$-con $2$-cc graph, then $$\gamma(G) \geq \left \lceil \frac{2}{3} \cdot \# L \right \rceil.$$
\end{theorem}

Before proving the result, we list some useful observations.  Let $G$ be a large $3$-con $2$-cc graph.
\begin{itemize}
    \item Every vertex of $G$ lies in at least one picture.
    \item Every vertex of $G$ lies on at least one and at most two tiles.
    \item A vertex of $G$ can dominate vertices in at most three different tiles. If $v$ lies on an intersection of two tiles, then $v$ can dominate only vertices in these two tiles.
    \item If $G$ is dominated, then there are no two consecutive tiles that include no vertex from the dominating set.
    \item No picture can be completely dominated by a vertex in the corner of its frame (because no picture has a diagonal).
\end{itemize}

\begin{proof}[Proof of Theorem~\ref{thm:dom-lower}]
Let $D$ be the smallest dominating set of graph $G$. The idea of the proof is to give each vertex $v$ from $D$ some charge which is transferred to tiles in which $v$ dominates some vertices. The transfer is done so that no charge is lost, thus the initial charge equals the charge in the end, enabling us to double count it, first by vertices, then by tiles.

Every vertex from $D$ receives a charge of $\frac{3}{2}$. The charge transfers in two phases based on the following rules. 

\begin{description}
\item[Phase 1] Let $v \in D$ and say that $v$ lies on a tile $T$.
\begin{itemize}
    \item If $v$ lies only on tile $T$, then $v$ sends a charge of $1$ to $T$. If $v$ dominates vertices in two other tiles, then each of them receives a charge of $\frac{1}{4}$. If $v$ dominates vertices in only one other tile, then it receives a charge of $\frac{1}{2}$. If $v$ dominates only vertices in $T$, then $T$ receives an additional charge of $\frac{1}{2}$.
    \item If $v$ lies in the intersection of two tiles, then each of them receives a charge of $\frac{3}{4}$. 
\end{itemize}
The charge of a tile after Phase~1 is the sum of charges it received from different vertices.

\item[Phase 2] If a tile $T$ has a charge strictly smaller than $1$ and $T$ got only a charge of $\frac{1}{4}$ from vertices in a neighboring tile $T'$, then the tile $T'$ sends a charge of $\frac{1}{4}$ to the tile $T$.
\end{description}

We need to argue that Phase~2 is well-defined, i.e., that $T'$ has enough charge to send some away. At the same time, we prove that the charge of the tile that gives some charge to its neighbor is at least $1$ even after Phase~2. 

Let $T$ be a tile that has a charge strictly less than $1$ after Phase~1, let $T'$ be a neighboring tile, and suppose that $T$ got a charge of $\frac{1}{4}$ from vertices in $T'$. This means that only one vertex from $T'$ sent some charge to $T$ in Phase~1, say vertex $v$. Since tile $T$ is dominated and contains no vertex from $D$, the picture of $T$ can only be $DD$, and $v$ dominates both vertices on the wall of $T$ neighboring $T'$. Since $T'$ is also dominated and $v$ cannot dominate the vertex $w$ in the opposite corner of $T'$, either $T'$ contains another vertex from $D$ or $w$ is dominated by a vertex from a neighboring tile $T'' (\neq T)$. 

In the first case, the charge of $T'$ after Phase~1 is at least $1 + \frac{3}{4} = \frac{7}{4}$ ($v$ lies only on $T'$, but the other vertex may lie on the intersection of $T'$ and $T''$). This means that the charge of $T'$ after Phase~2 is at least $1$ (even if $T'$ sends a charge of $\frac{1}{4}$ to each one of its two neighbors). In the second case, the charge of $T'$ after Phase~1 is at least $1 + \frac{1}{4} = \frac{5}{4}$ (from the vertex $v$ and the vertex dominating $w$). Since in this case, $T''$ contains a vertex from $D$, its charge after Phase~1 is at least $1$, so in Phase~2 $T'$ sends a charge of $\frac{1}{4}$ only to $T$, thus the charge of $T'$ after Phase~2 is at least $1$.

Next, we prove the following.

\begin{claim}
After Phase~2, every tile has a charge of at least $1$.
\end{claim}

\begin{myproof}
Phase~2 does not reduce the charge of a tile below $1$, so tiles that have a charge at least $1$ after Phase~1 also have a sufficiently large charge after Phase~2.

If a tile includes a vertex from $D$ which does not lie on another tile, then its charge is at least $1$. If a tile $T$ includes a vertex $v$ from $D$ which lies on an intersection of two tiles, then $v$ lies in a corner, thus another vertex is needed to dominate the whole tile $T$. Hence the charge of $T$ after Phase~1 is at least $\frac{3}{4} + \frac{1}{4} = 1$. If a tile $T$ does not contain any vertex from $D$, then, since it is dominated, it must receive a charge of at least $\frac{1}{4}$ from vertices in each of its neighbors in Phase~1. If it receives only a charge of $\frac{1}{4}$ from one side, then in Phase~2 it receives another charge of  $\frac{1}{4}$ from this tile. Thus $T$ receives a charge of at least $\frac{1}{2}$ from each of its neighbors after both phases are over, resulting in the charge of $T$ being at least $1$. 
\end{myproof}

Since the initial charge equals the charge in the end, and $|D| = \gamma(G)$, we get
\begin{align*}
    \frac{3}{2} \cdot |D| & \geq  1 \cdot \# L \\
    \gamma(G) & \geq  \frac{2}{3} \cdot \# L
\end{align*}
As the domination number of $G$ is an integer, it follows that $\gamma(G) \geq \left \lceil \frac{2}{3} \cdot \# L \right \rceil$.
\end{proof}

The lower bound from Theorem~\ref{thm:dom-lower} is sharp, which can be seen in the following example.

\begin{example}
Let $G_3 = n \cdot DDLDDLAIVL$, where $n \geq 1$ is an odd number. Figure~\ref{fig:G3} shows the dominating set of size $\frac{2}{3} \cdot 3 \cdot n$, meaning $\gamma(G_3) \leq 2 n$. Our formula from Theorem~\ref{thm:dom-lower} shows the same, as $\left \lceil \frac{2}{3} \cdot \#L \right \rceil = \left \lceil \frac{2}{3} \cdot 3 \cdot n \right \rceil = 2n$.


\begin{figure}[h]
    \centering
    \includegraphics[width=\textwidth]{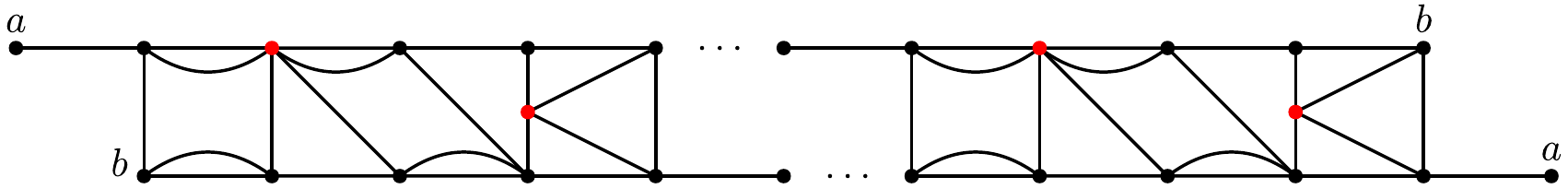}
    \caption{Graph $G_3$ with a marked dominating set of size $2n$.}
    \label{fig:G3}
\end{figure}

Assume $\gamma(G_3) < 2n$. Then there exists at least one trinity of consecutive pictures $DDLDDLAIVL$, which contains at most one vertex from the dominating set. No vertex of graph $G_3$ dominates all the inner vertices of the trinity, which are marked orange in Figure~\ref{fig:G3_trinity}, meaning we need at least two vertices to dominate this trinity of pictures. Therefore $\gamma(G_3) \geq 2n$.


\begin{figure}[h]
    \centering
    \includegraphics[width=0.35\textwidth]{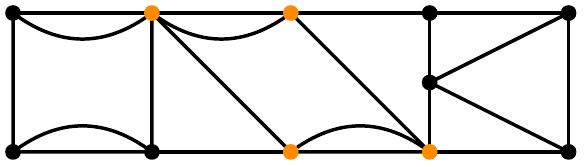}
    \caption{Trinity of consecutive pictures $DDLDDLAIVL$, where we want the inner vertices, marked orange, to be dominated by one vertex. Note that none of these vertices can be dominated by a vertex outside of this trinity of tiles.}
    \label{fig:G3_trinity}
\end{figure}

From this follows that $\gamma(G_3) = 2n$.
\end{example}

\section{Independence number}
\label{sec:independence}

In this section, we present the sharp upper and lower bounds for the independence number of large $3$-con $2$-cc graphs.

\subsection{Upper bound}
\label{sec:independence-upper}

\begin{theorem} \label{thm:indep-upper}
If $G$ is a large $3$-con $2$-cc graph, then
$$\alpha(G) \leq \left \lfloor{\frac{| \textsc{V}(G)|}{2}}\right \rfloor.$$
\end{theorem}

\begin{proof}
Since all large $3$-con $2$-cc graphs are Hamiltonian~\cite{ham}, and the independence number of Hamiltonian graphs is at most $\frac{1}{2}|\textsc{V}(G)|$, we obtain the desired upper bound.
\end{proof}

The following example shows that the upper bound from Theorem~\ref{thm:indep-upper} is sharp.

\begin{example}
Let $G_4 = n \cdot HdL$, where $n \geq 3$ is an odd number. Then $| \textsc{V}(G_4)| = 6n$. 
Figure~\ref{fig:G4} shows that we can choose $3n$ independent vertices from the graph $G_4$, meaning $\alpha(G_4) \geq 3n$. 


\begin{figure}[h]
    \centering
    \includegraphics[width=0.77\textwidth]{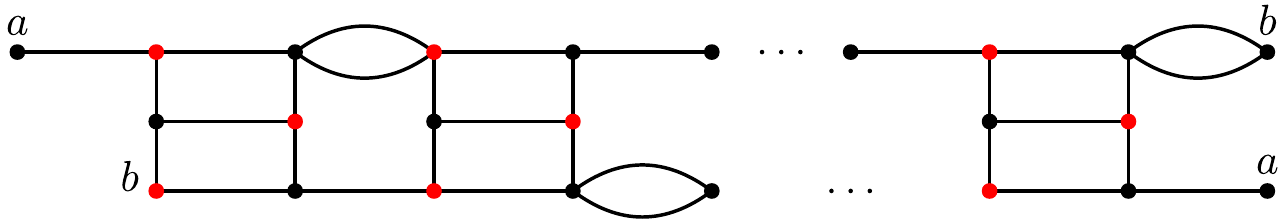}
    \caption{Graph $G_4$ with a marked independent set of size $3n$.}
    \label{fig:G4}
\end{figure}

Every vertex of graph $G_4$ lies in exactly one picture. Since we can choose at most three independent vertices in each of the $n$ pictures, $\alpha(G_4) \leq 3n$, which is also the result of Theorem~\ref{thm:indep-upper}. Therefore $\alpha(G_4) = 3n = \left \lfloor{\frac{| \textsc{V}(G_4)|}{2}}\right \rfloor$.
\end{example}

\subsection{Lower bound}
\label{sec:independence-lower}

\begin{theorem} \label{thm:indep-lower}
If $G$ is a large $3$-con $2$-cc graph, then
$$\alpha(G) \geq \min\{\#L + \#d, 2 \cdot \#L - 1\}.$$
\end{theorem}

\begin{proof}
For every large $3$-con $2$-cc graph $G$ we can construct a graph $G'$ from the same frames used for $G$, without using the pictures.
We notice that if we add pictures into the frames in $G'$ to get the original graph $G$, we only add vertices and do not connect any vertices that were previously not connected, thus any picture we add can only increase the independence number,  therefore $\alpha(G) \geq \alpha(G')$. Note that the graph $G'$ will have the same number of frames $L$ and $dL$ as the initial graph $G$. Therefore it suffices to prove the proposed lower bound for the graph $G'$.

We separate two cases, the first case is if there are only $dL$ frames and the second if there is at least one $L$ frame.

\begin{description}
\item[Case 1]
If there are only $dL$ frames, we can find $2 \cdot \#L - 1$ independent vertices, as is shown in Figure~\ref{fig:G'_case1}. Note that in this case $\min\{\#L + \#d, 2 \cdot \#L - 1\} = 2 \cdot \#L - 1$.



\begin{figure}[h]
    \centering
    \includegraphics[width=0.8\textwidth]{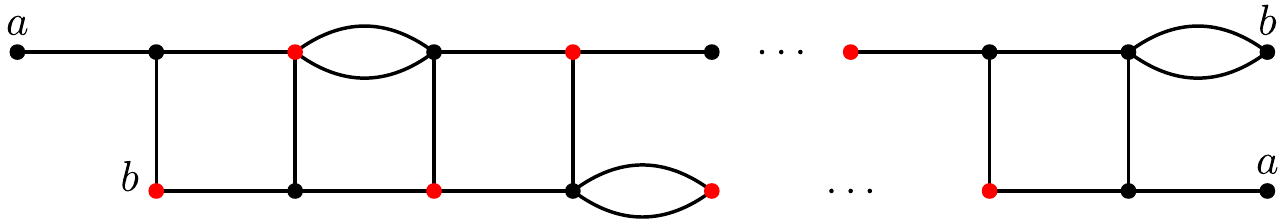}
    \caption{Graph $G'$ from Case~1 with a marked independent set of size $2 \cdot \#L - 1$.}
    \label{fig:G'_case1}
\end{figure}

\item[Case 2]
If there is at least one $L$ frame, then we can choose the independent set based on the following method. Note that double edges can be ignored when studying the independence number. The graph $G'$ is then composed of $3$- and $4$-cycles, which are connected with additional edges (marked orange in Figure~\ref{fig:G'_case2}). These additional edges come from where the top and bottom paths of the pictures were in $G$. To obtain an independent set of appropriate size, we select one vertex from each $3$-cycle and two vertices from each $4$-cycle. For every $3$-cycle, we select the vertex of degree $3$ on its left side. If we have two consecutive $3$-cycles, the vertices we chose from them are independent. When selecting vertices in the $4$-cycles, we consider all consecutive $4$-cycles between two $3$-cycles and select vertices for the independent set in these $4$-cycles from right to left. The $3$-cycle on the right of the consecutive $4$-cycles determines how we choose the independent set in the right-most $4$-cycle, which in turn uniquely determines how we select two independent vertices in each of these $4$-cycles (in the same manner as in Figure~\ref{fig:G'_case1}). Notice that the $3$-cycle on the left of these $4$-cycles gives no restriction on the selected vertices.


\begin{figure}[h]
    \centering
    \includegraphics[width=0.8\textwidth]{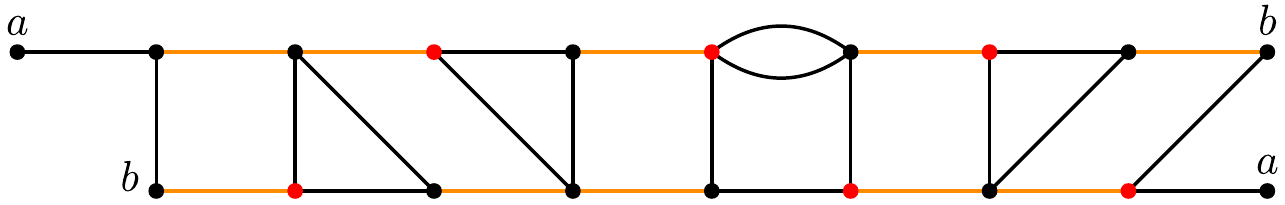}
    \caption{An example of the graph $G'$ from Case~2 with a marked independent set of size $\#L + \#d$. The edges that connect the $3$- and $4$-cycles are marked orange.}
    \label{fig:G'_case2}
\end{figure}



We have thus chosen two vertices in each $4$-cycle and one vertex in each $3$-cycle. Since the number of $3$-cycles is $\#L - \#d$ and the number of $4$-cycles is $\#d$, we have found an independent set of size $(\#L - \#d) + 2 \cdot \#d = \# L + \# d $. Note that in this case $\min\{\#L + \#d, 2 \cdot \#L - 1\} = \# L + \# d$. \qedhere
\end{description}
\end{proof}

The following two examples show that the lower bound from Theorem~\ref{thm:indep-lower} is sharp. The first example naturally follows from the proof of Theorem~\ref{thm:indep-lower}, while the second example provides a non-trivial family of sharpness examples. Additionally, examples are selected in such a way that different parts of the minimum are attained.


\begin{example}
Let $G_5$ be a large $3$-con $2$-cc graph built from tiles $DDdL$ and $DDL$, so that not all of the tiles are $DDdL$.

From Theorem~\ref{thm:indep-lower} we know that $\alpha(G_5) \geq \#L + \#d$. 
Similarly as in the proof, we can find $\#d$ $4$-cycles and $\#L-\#d$ $3$-cycles in $G_5$, so that every vertex lies on exactly one of them. Every $4$-cycle is formed by the two vertices on the right of a $DDdL$ tile and the two vertices on the left of the next tile to the right. Every $3$-cycle is formed by the two vertices on the right of a $DDL$ tile and the two vertices on the left of the next tile. Two of those vertices are identified, thus giving us a $3$-cycle. The $3$-cycles and $4$-cycles are marked in Figure~\ref{fig:G5}. 



\begin{figure}[h!]
    \centering
    \includegraphics[width=\textwidth]{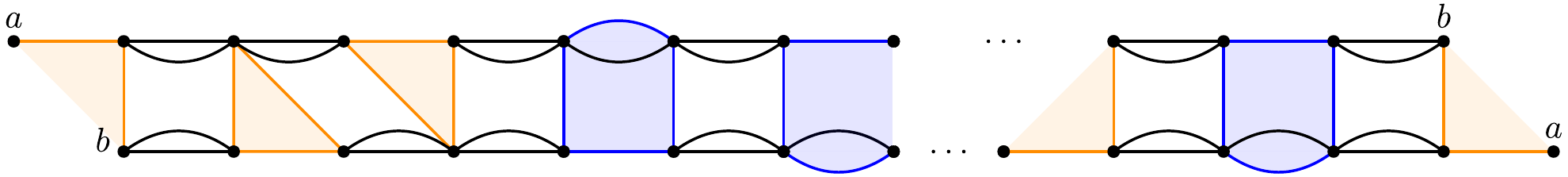}
    \caption{Graph $G_5$ with marked $3$-cycles and $4$-cycles.}
    \label{fig:G5}
\end{figure}

We can choose at most one independent vertex from every $3$-cycle and at most two independent vertices from every $4$-cycle, therefore $\alpha(G_5) \leq 2 \cdot \#d + 1 \cdot (\#L - \#d) = \#L + \#d$.

From this it follows that $\alpha(G_5) = \#L + \#d$.
\end{example}

\begin{example}
Let $G_6$ be a large $3$-con $2$-cc graph that is built from $DDdL$, $VIAdL$, and $AIVdL$ tiles, but not all tiles are $VIAdL$, and not all tiles are $AIVdL$.

From Theorem~\ref{thm:indep-lower} we know that $\alpha(G_6) \geq 2 \cdot \#L - 1$. We can find at most two independent vertices in each of the tiles $DDdL$, $VIAdL$, and $AIVdL$, therefore we can find at most $2 \cdot \#L$ independent vertices in $G_6$.

For contradiction suppose that $\alpha(G_6) \neq 2 \cdot \#L - 1$, meaning $\alpha(G_6) = 2 \cdot \#L$. We try to construct an independent set $A$ with $2 \cdot \#L$ vertices. Set $A$ must include exactly two vertices from every tile because otherwise set $A$ would have to include at least $3$ vertices from one tile, which is impossible. 

There are two different ways in which we can choose two independent vertices from a $DDdL$ tile, and three different ways for tiles $VIAdL$ and $AIVdL$. All options are shown in Figure~\ref{fig:G6_tiles}.


\begin{figure}[h!]
    \centering
    \includegraphics[width=\textwidth]{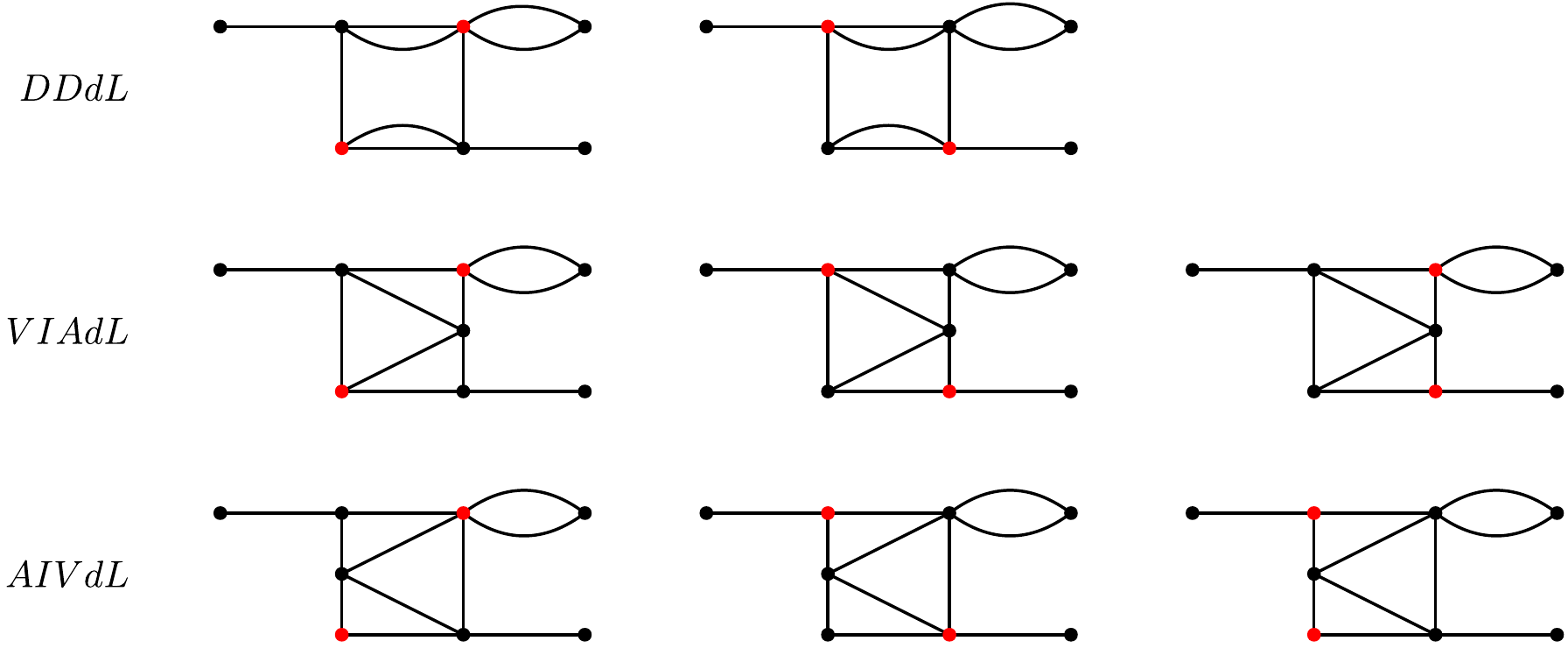}
    \caption{Tiles $DDdL$, $VIAdL$, and $AIVdL$ with two independent vertices marked.}
    \label{fig:G6_tiles}
\end{figure}

Even though tiles $VIAdL$ and $AIVdL$ have a third option for the choice of two independent vertices (where the selected vertices are not diagonal), we can't choose the vertices in set $A$ in this way, since we know that we have to choose two independent vertices from every tile. If we choose the top and bottom right vertex in a $VIAdL$ tile, then the only way to choose two vertices in the next tile is if that tile is also a $VIAdL$ tile and we choose the top and bottom right vertices. We continue this for all tiles, but since not all tiles are $VIAdL$, at some point we are not able to choose two independent vertices in the next tile. For the same reason, we also cannot choose the two vertices on the left of an $AIVdL$ tile.

This means that for all tiles, the two vertices that are included in set $A$ are the diagonal ones, without loss of generality we can assume that those diagonal vertices in the first tile are the bottom left and the top-right vertex. This choice determines which vertices we must choose in the tile to the right and so on, as is shown in Figure~\ref{fig:G6}.


\begin{figure}[h]
    \centering
    \includegraphics[width=0.8\textwidth]{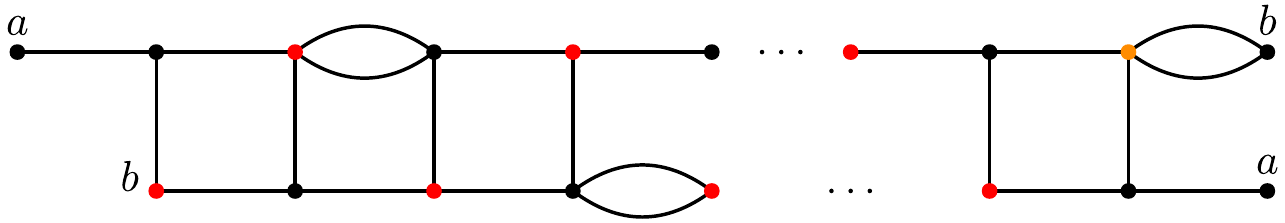}
    \caption{Graph $G'_6$ was constructed from the same frames used for $G_6$, without using the pictures. The first tile of graph $G'_6$ determines which two vertices are included in set $A$ for all other tiles. When we get to the last tile, we get a contradiction (marked orange).}
    \label{fig:G6}
\end{figure}

When we get to the last tile we get a contradiction. Because of the tile to the left, the only possible vertices from the last tile that can be included in $A$ are the bottom left and the top-right vertex. But the top-right vertex is connected to a vertex in the first tile that is already included in set $A$, therefore set $A$ cannot include two vertices from the last tile. 

This means that the independent set $A$ that has $2 \cdot \#L$ elements cannot exist and $\alpha(G_6) = 2 \cdot \#L - 1$.
\end{example}





\section*{Acknowledgements}
The authors were introduced to the structure of $2$-crossing-critical graphs in a workshop at the University of Ljubljana, organized by prof.\ dr.\ Drago Bokal. We thank him for several illuminating conversations and ideas. We would also like to thank Sandi Klavžar, Alen Vegi Kalamar, and Simon Brezovnik for co-organizing the workshop.

V.I.\ was supported by a postdoctoral fellowship at the Simon Fraser University (Canada) and by the Slovenian Research Agency (research core funding P1-0297 and projects J1-2452, J1-1693, N1-0095, N1-0218).


\end{document}